\documentclass[12pt,twoside=semi,leqno,abstracton
]{scrartcl}
\pdfoutput=1
\setkomafont{title}{\LARGE\bfseries\boldmath}
\setkomafont{section}{\large\bfseries\boldmath}
\setkomafont{subsection}{\large\bfseries\itshape\boldmath}
\setkomafont{subtitle}{\itshape\boldmath}
\setkomafont{disposition}{\normalfont\normalcolor\bfseries}
\setkomafont{dictum}{\scshape}
\setkomafont{paragraph}{\bfseries\boldmath}
\setkomafont{pagehead}{\normalfont\upshape\small}
\setkomafont{pagenumber}{\normalfont\upshape\normalsize}
\usepackage[arxiv%
]{MyKoma2}
\usepackage{mathtools}
\makeatletter

\DeclareSymbolFont{symbolsC}{U}{pxsyc}{m}{n}
\DeclareMathSymbol{\medcirc}{\mathbin}{symbolsC}{7}
\DeclareMathSymbol{\medbullet}{\mathbin}{symbolsC}{8}

\usepackage{accents}

\DeclareFontFamily{U}{mathb}{\hyphenchar\font45}
\DeclareFontShape{U}{mathb}{m}{n}{
<-6> mathb5 <6-7> mathb6 <7-8> mathb7
<8-9> mathb8 <9-10> mathb9
<10-12> mathb10 <12-> mathb12
}{}
\DeclareSymbolFont{mathb}{U}{mathb}{m}{n}
\DeclareMathSymbol{\llcurly}{\mathrel}{mathb}{"CE}
\DeclareMathSymbol{\ggcurly}{\mathrel}{mathb}{"CF}
\makeatother

\begin{document}

\title{Convex functions on dual Orlicz spaces}
\author[1]{Freddy Delbaen}
\address{Department of Mathematics, ETH Zürich \&  \newline
Institute of Mathematics, University of Zürich
}
\email{delbaen@math.ethz.ch}
\thanks{${}^1$Part of this work was done while the first named author was
  on visit at Tokyo Metropolitan University.}
\author[2]{Keita Owari}
\address{Department of Mathematical Sciences, Ritsumeikan
  University
}
\thanks{${}^2$Supported in part by JSPS grant number JP17K14210.}
\email{owari@fc.ritsumei.ac.jp}

\runauthor{F. Delbaen and K. Owari}
\runtitle{Convex Functions on Dual Orlicz Spaces}%
\date{}

\ArXiv{1611.06218}


\abstract{%
  In the dual $L_{\Phi^*}$ of a $\Delta_2$-Orlicz space $L_\Phi$, that we call a dual Orlicz space, we
  show that a proper (resp. finite) convex function is lower
  semicontinuous (resp. continuous) for the Mackey topology
  $\tau(L_{\Phi^*},L_\Phi)$ if and only if on each order interval
  $[-\zeta,\zeta]=\{\xi: -\zeta\leq \xi\leq\zeta\}$
  ($\zeta\in L_{\Phi^*}$), it is lower semicontinuous
  (resp. continuous) for the topology of convergence in
  probability. For this purpose, we provide the following Komlós type
  result: every norm bounded sequence $(\xi_n)_n$ in $L_{\Phi^*}$
  admits a sequence of forward convex combinations
  $\bar\xi_n\in\conv(\xi_n,\xi_{n+1},...)$ such that
  $\sup_n|\bar\xi_n|\in L_{\Phi^*}$ and $\bar\xi_n$ converges a.s.}

\maketitle

\section{Introduction}
\label{sec:ConvDualBanach}

\noindent\textbf{Notation.} We use the usual  probabilistic notation. 
$(\Omega,\FC,\PB)$ is a probability space and
$L_0:=L_0(\Omega,\FC,\PB)$ stands for the space of (classes modulo
equality $\PB$-a.s. of) finite measurable functions equipped with the
complete metrisable vector topology $\tau_{L_0}$ of convergence in
$\PB$ (in probability). As usual, we identify a measurable function
with the class it generates.  We write $\EB[\xi]:=\int_\Omega\xi d\PB$
whenever it makes sense, and $L_p:=L_p(\Omega,\FC,\PB)$,
$p\in [1,\infty]$, denote the standard Lebesgue spaces.

Problems in financial mathematics often involve convex functions on
the dual $E'$ of a Banach space $E$ (see Section~\ref{sec:MonUtil} for
a motivating example). Dealing with such $f$, the lower semicontinuity
(lsc) and continuity for the Mackey topology $\tau(E',E)$ are basic;
the former ($\Leftrightarrow$ $\sigma(E',E)$-lsc) is necessary and
sufficient (by the Hahn-Banach theorem) for the dual representation
\begin{align*}
  f(x')=\sup_{x\in E}(\langle x,x'\rangle-f^*(x)),\,x'\in E';\quad \text{where }f^*(x)=\sup_{x'\in E'}(\langle x,x'\rangle-f(x'))
\end{align*}

Generally speaking, $\tau(E',E)$ is not easy to deal with, but its
restrictions to bounded sets often have a nice description. The best
known case is $L_\infty=L_1'$: on bounded sets, $\tau(L_\infty,L_1)$
coincides with the topology of $L_0$, a fortiori metrisable (this
result is due to Grothendieck; see \citep{MR0372565}, pp.222-223).
Hence by the Krein-Šmulian theorem, we have

\begin{proposition}
  \label{prop:LinftyLSC1}
  For  proper convex functions $f$ on $L_\infty$, the following are
  equivalent:
  \begin{enumerate}
  \item $f$ is $\sigma(L_\infty,L_1)$-lsc, equivalently
    $\tau(L_\infty,L_1)$-lsc;
  \item $f$ is sequentially $\tau(L_\infty,L_1)$-lsc;
  \item $f$ is lsc on bounded sets for the topology of convergence in
    probability.
  \end{enumerate}
\end{proposition}

The following result for the $\tau(L_\infty,L_1)$-continuity is also
known for convex risk measures (e.g. \citep{MR2277714,MR2648597}), and
it remains true for finite convex functions; but we could not find a
relevant reference, so we include a short proof in the Appendix.
\begin{proposition}
  \label{prop:ContiLinfty1}
  For any convex function $f:L_\infty\rightarrow\RB$, the following
  are equivalent:
  \begin{enumerate}
  \item $f$ is $\tau(L_\infty,L_1)$-continuous;
  \item $f$ is \textbf{sequentially} $\tau(L_\infty,L_1)$-continuous;
  \item $f$ is continuous for the topology of convergence in
    probability on bounded sets.
  \end{enumerate}
\end{proposition}

Let $\Phi:\RB\rightarrow\RB$ be a (finite coercive)
\emph{\bfseries Young function}, i.e. an even convex function with
$\Phi(0)=0$ and
$\lim_{x\rightarrow+\infty}\frac{\Phi(x)}x=+\infty$. Then
$\BB_\Phi:=\{\xi\in L_0: \EB[\Phi(\xi)]\leq 1\}$ is a closed convex
solid subset of $L_0$ bounded in $L_1$ containing a non-zero constant,
thus it generates a Banach lattice with the closed unit ball
$\BB_\Phi$, called the \emph{\bfseries Orlicz space}:
\begin{align*}
  L_\Phi:=\mcup_{\lambda>0}\lambda\BB_\Phi=\{\xi\in L_0: \exists \lambda>0\text{ with }\EB[\Phi(\lambda\xi)]<\infty\},
\end{align*}
given the norm $\|\xi\|_\Phi:=\inf\{\lambda>0:\xi\in \lambda
\BB_\Phi\}$ and \emph{\bfseries a.s. pointwise order}. In general,
$L_\infty\subset L_\Phi\subset L_1$ with continuous injections.  The
conjugate $\Phi^*(y):=\sup_x(xy-\Phi(x))$ is again a (finite coercive)
Young function, so the Orlicz space $L_{\Phi^*}$ is similarly defined.

A Young function $\Phi$ is said to satisfy the
\emph{\bfseries\boldmath$\Delta_2$-condition}, denoted by $\Phi\in
\Delta_2$, if $\limsup_{x\rightarrow\infty}\Phi(2x)/\Phi(x)<\infty$,
or equivalently
\begin{align}
  \label{eq:Delta2Equiv}
  p_\Phi:=\inf_{x\geq 0}p_\Phi(x):=\inf_{x\geq 0}\left(\sup_{y>x}\frac{y\Phi'(y)}{\Phi(y)}\right)<\infty,
\end{align}
where $\Phi'$ is the left-derivative of $\Phi$ (see \citep{MR1113700},
Th.~II.2.3). If $\Phi\in\Delta_2$, the dual $L_\Phi'$ of $L_\Phi$ is
identified via $\langle \xi,\eta\rangle=\EB[\xi\eta]$ with
$L_{\Phi^*}$ given an equivalent norm
$\|\xi\|_{(\Phi^*)}:= \sup_{\eta\in \BB_\Phi}\EB[\eta\xi]$; more
precisely
$\|\xi\|_{\Phi^*} \leq \|\xi\|_{(\Phi^*)}\leq 2\|\xi\|_{\Phi^*}$, and
$\EB[\eta\xi]\leq \|\eta\|_\Phi\|\xi\|_{(\Phi^*)}$.
In particular, $L_\Phi$ is reflexive if both $\Phi,\Phi^*\in\Delta_2$;
the condition is also necessary if $(\Omega,\FC,\PB)$ is atomless. 
In the sequel, we suppose $\Phi\in\Delta_2$ unless otherwise
mentioned.

Our basic interest is to understand the
$\tau(L_{\Phi^*},L_\Phi)$-lower semicontinuity and continuity of
convex functions through the sequential convergence in probability on
bounded sets. At this point, we note that there are two possible
interpretations of ``bounded sets''; norm bounded sets, and
\emph{\bfseries order bounded} sets, that is, those
$A\subset L_{\Phi^*}$ contained in an \emph{\bfseries order interval}
$[-\zeta,\zeta]:=\{\xi: -\zeta\leq \xi\leq \zeta\}$,
$0\leq \zeta\in L_{\Phi^*}$, i.e.
\emph{\bfseries dominated in $L_{\Phi^*}$}. Since
$[-\zeta,\zeta]\subset\|\zeta\|_{\Phi^*}\BB_{\Phi^*}$, the order
bounded sets are norm bounded, and in $L_\infty$, the two notions of
boundedness are identical.
  
The core of this papar is a few variants of
\emph{\bfseries Komlós's theorem} in the dual $L_{\Phi^*}$ of a
$\Delta_2$-Orlicz space $L_\Phi$. The classical Komlós theorem
\citep{MR0210177} states that any bounded sequence $(\xi_n)_n$ in
$L_1$ has a subsequence $(n_k)_k$ as well as $\xi\in L_1$ such that
for any further subsequence $(n_{k(i)})_i$, the
\emph{\bfseries Cesàro means} $\frac1n\sum_{i\leq N}\xi_{n_{k(i)}}$
converges a.s. to $\xi$. The basic form of our variants
(Theorem~\ref{thm:Cesaro1}) asserts that under the stronger assumption
of boundedness in $L_{\Phi^*}$ and convergence in $\PB$, a subsequence
can be chosen so that the Cesàro means are
\emph{\bfseries order bounded} in $L_{\Phi^*}$ as well. Its
practically useful consequence (Corollary~\ref{cor:ConvCombUseful}) is
that any \emph{\bfseries norm bounded} sequence in $L_{\Phi^*}$, not
necessarity convergent in $\PB$, has an \emph{\bfseries order bounded}
(and a.s. convergent) sequence of
\emph{\bfseries forward convex combinations} 
$\zeta_n\in\conv(\xi_k;k\geq n)$, $n\geq 1$.  Moreover, if
$(\Omega,\FC,\PB)$ is atomless, this version of Komlós theorem
characterises the $\Delta_2$-Orlicz spaces
(Theorem~\ref{thm:Delta2Ch1}).

In view of the Krein-Šmulian theorem, this form of Komlós theorem
yields that a \emph{\bfseries convex} set $C\subset L_{\Phi^*}$ is
$\sigma(L_{\Phi^*},L_\Phi)$-closed if (and only if\footnotemark) it is
\emph{\bfseries order closed}:
\footnotetext{Regardless of $\Phi\in \Delta_2$ and convexity,
  $\sigma(L_{\Phi^*},L_\Phi)$-closed $\Rightarrow$ order closed
  $\Rightarrow$ norm closed since $L_\Phi$ is identified with the
  \emph{\bfseries order continuous dual of $L_{\Phi^*}$} and norm
  convergent sequences have order convergent subsequences; see
  e.g. \citep[Ch.~14]{MR704021} for details and unexplained
  terminologies.}
\begin{align}
  \label{eq:OrderClosed1}
  \forall \zeta\in L_{\Phi^*},\, C\cap [-\zeta,\zeta]\text{ is closed in }L_0.
\end{align}
In terms of functions, this reads as: a proper convex function $f$ on
$L_{\Phi^*}$ is $\sigma(L_{\Phi^*},L_\Phi)$-lsc 
if (and only if) $f$ is lsc for the topology of $L_0$ on order
intervals in $L_{\Phi^*}$, or explicitly $f(\xi)\leq
\liminf_nf(\xi_n)$ whenever $\xi_n\rightarrow \xi$ in $\PB$ and
$\sup_n|\xi_n|\in L_{\Phi^*}$ (Theorem~\ref{thm:DualRep1}).  A similar
characterisation of the $\tau(L_{\Phi^*},L_\Phi)$-continuity is also
given (Theorem~\ref{thm:MackeyConti}).


The question of the weak* closedness of order closed convex sets in
$L_{\Phi^*}$ is raised by \citep{MR2648595} in the context of
representation of convex risk measures. They claimed in
\citep[Lemma~6]{MR2648595} that this is the case because
$\sigma(L_{\Phi^*},L_\Phi)$ has the following property:
\begin{equation}
  \label{eq:CProp}
  \tag{$C$}
  \begin{minipage}[c]{0.85\linewidth}
    if $\xi_\alpha\rightarrow\xi$ in $\sigma(L_{\Phi^*},L_\Phi)$,
    there exist a sequence of indices $(\alpha_n)_n$ and
    $\zeta_n\in\conv(\xi_{\alpha_k}; k\geq n)$, $n\geq 1$, such that
    $\zeta_n\rightarrow \xi$ a.s. and $\sup_n|\zeta_n|\in L_{\Phi^*}$.
  \end{minipage}
\end{equation}
Unfortunately, this is not correct; (\ref{eq:CProp}) holds (if and)
\emph{\bfseries only if} $L_\Phi$ is reflexive
(\citep{Gao_Xanthos15:_c}). For $(\zeta_n)_n$ in (\ref{eq:CProp})
converges in $\sigma(L_{\Phi^*},L_\Phi)$, thus (\ref{eq:CProp}) would
imply that for any convex set $C\subset L_{\Phi^*}$, its weak* closure
coincides with the \emph{\bfseries sequential weak* closure}
$C_{(1)}:=\{\xi: \xi=w^*\text{-}\lim_{n}\xi_n
{\text{ with }}(\xi_n)_n\subset C\}$,
while any non-reflexive Banach space has a convex set $C$ in the dual
such that $C_{(1)}$ is not weak* closed
(\citep[Th.~2]{MR2943053}; 
see \citep{MR1901532} for the history of problem of sequential weak*
closures which goes back to Banach \citep{banach32:_theor}).
On the other hand, Corollary~\ref{cor:ConvCombUseful} shows that the
property (\ref{eq:CProp}) holds for \emph{\bfseries bounded nets}
(recall that convergent nets need not be bounded). 

\section{Mackey Topology on Orlicz Spaces}
\label{sec:MackeyOrlicz}

The following criterion for $\sigma(L_\Phi,L_{\Phi^*})$-compact sets
is known (e.g.  \citep{MR1113700}, Th.~IV.5.1), but we include a short
proof in the Appendix. Here the $\Delta_2$-condition is not necessary. \nolinebreak
\begin{lemma}
  \label{lem:CompactSets}
  (Regardless of $\Phi\in\Delta_2$,) a set $A\subset L_\Phi$ is
  relatively $\sigma(L_\Phi,L_{\Phi^*})$-compact if and only if for
  each $\xi\in L_{\Phi^*}$, $A\xi:=\{\eta\xi: \eta\in A\}$ is
  uniformly integrable.
\end{lemma}

\begin{lemma}
  \label{lem:MackeyOrliczRelations}
  $\tau(L_{\Phi^*},L_\Phi)$ is finer than the restriction of
  $\tau_{L_0}$ to $L_{\Phi^*}$, and
  \begin{equation}
    \label{eq:MackeyOrderInterval}
    \forall\zeta\in L_{\Phi^*},\,    \tau(L_{\Phi^*},L_\Phi)|_{[-\zeta,\zeta]}=\tau_{L_0}|_{[-\zeta,\zeta]}.
  \end{equation}
  In particular, $\tau(L_{\Phi^*},L_\Phi)$ is metrisable on order
  bounded sets.  If $\Phi\in\Delta_2$, we have
  \begin{equation}
    \label{eq:MackeyBall}
    \sigma(L_{\Phi^*},L_\Phi)|_{\BB_{\Phi^*}}\subset \tau_{L_0}|_{\BB_{\Phi^*}}\subset \tau(L_{\Phi^*},L_\Phi)|_{\BB_{\Phi^*}}.
  \end{equation}  
\end{lemma}
\begin{proof}
  The (image in $L_\Phi$) of $\BB_{L_\infty}$ is
  $\sigma(L_\Phi,L_{\Phi^*})$-compact, thus defines a Mackey
  continuous seminorm $\xi\mapsto
  \sup_{\eta\in \BB_{L_\infty}}|\EB[\xi\eta]|=\EB[|\xi|]\geq
  \EB[|\xi|\wedge 1]$, so $\tau(L_{\Phi^*},L_\Phi)$ is finer than the
  restriction of $\tau_{L_0}$. On the other hand, for any
  $\sigma(L_\Phi,L_{\Phi^*})$-compact set $A\subset L_\Phi$ and
  $\zeta\in L_{\Phi^*}$, one has
  $\lim_N\sup_{\eta\in A}\PB(|\eta|\vee|\zeta|>N)= 0$, and
  $p_A(\xi):=\sup_{\eta\in A}|\EB[\eta\xi]| \leq
  \sup_{\eta\in
    A}\EB[|\eta\zeta|\ind_{\{|\eta|\vee|\zeta|>N\}}]+N^2\EB[|\xi|\wedge
  1]$, $\forall N\in\NB$, on $[-\zeta,\zeta]$. A standard
  diagonalisation procedure then shows that $p_A$ is
  $\tau_{L_0}$-continuous on $[-\zeta,\zeta]$, and we see that
  $\tau(L_{\Phi^*},L_\Phi)|_{[-\zeta,\zeta]}\subset
  \tau_{L_0}|_{[-\zeta,\zeta]}$.  Finally, if $\Phi\in \Delta_2$, so
  $L_{\Phi^*}=L_\Phi'$, $\BB_{\Phi^*}$ is
  $\sigma(L_{\Phi^*},L_\Phi)$-compact, thus $\eta\BB_{\Phi^*}$,
  $\eta\in L_\Phi$, are uniformly integrable. Thus $\xi_n\in
  \BB_{\Phi^*}$ and $\xi_n\rightarrow\xi$ in $\PB$ imply
  $\EB[\eta\xi_n]\rightarrow\EB[\eta\xi]$ ($\forall \eta\in L_\Phi$),
  i.e. $\xi_n\rightarrow\xi$ in $\sigma(L_{\Phi^*},L_\Phi)$, which
  proves (\ref{eq:MackeyBall}).
\end{proof}


In the last part, the assumption $\Phi\in\Delta_2$ is used only to
ensure that bounded sequences are relatively
$\sigma(L_{\Phi^*},L_\Phi)$-compact. Thus the same argument shows that:
\begin{corollary}
  \label{cor:NullInPandWeakStarConv}
  (Regardless of $\Phi\in\Delta_2$,) if a sequence $(\xi_n)_n$ in
  $L_{\Phi^*}$ is null in $\PB$ and converges in
  $\sigma(L_{\Phi^*},L_\Phi)$ to $\xi$, then $\xi=0$.
\end{corollary}

\begin{remark}
  \label{rem:MackeyBDD1}
  On $\BB_{\Phi^*}$, $\tau(L_{\Phi^*},L_\Phi)$ is not generally the
  same as the topology of $L_0$. For example, if $A_n\in\FC$ are
  disjoint with $\PB(A_n)>0$, $\xi_n=\PB(A_n)^{-1/2}\ind_{A_n}$ form a
  sequence in $\BB_{L_2}$, null in $\PB$, but $\|\xi_n\|_2\equiv 1$,
  while $\tau(L_2,L_2)$ is the norm topology.
\end{remark}

\begin{proposition}
  \label{prop:MackeyNorm}
  If $\Phi\in\Delta_2$, the following are equivalent for all convex
  $C\subset L_{\Phi^*}$:
  \begin{enumerate}
  \item $C$ is $\sigma(L_{\Phi^*},L_\Phi)$-closed;
  \item $C$ is \textbf{sequentially}
    $\sigma(L_{\Phi^*},L_\Phi)$-closed;
  \item for each $\lambda>0$, $C\cap\lambda\BB_{\Phi^*}$ is closed in
    $L_0$, or equivalently, $\xi_n\in C$ ($\forall n$),
    $\xi_n\rightarrow\xi$ in $\PB$ and
    $\sup_n\|\xi_n\|_{\Phi^*}<\infty$ imply $\xi\in C$.
  \end{enumerate}

\end{proposition}
\begin{proof}
  By the Krein-Šmulian theorem, (1) $\Leftrightarrow$
  $C\cap\lambda\BB_{\Phi^*}$, $\lambda>0$, are
  $\sigma(L_{\Phi^*},L_\Phi)$-closed, and the three kinds of
  closedness are the same for $C\cap\lambda\BB_{\Phi^*}$ by
  (\ref{eq:MackeyBall}).   
\end{proof}


\section{Komlós-Type Results}
\label{sec:Komlos}

\emph{In the sequel, we suppose $\Phi\in\Delta_2$ so that
  $L_{\Phi^*}=L_\Phi'$ unless otherwise mentioned.}

 Recall that
$\Phi\in\Delta_2$ if and only if
$p_\Phi=\inf_{x\geq 0}p_\Phi(x)<\infty$ where
$p_\Phi(x)=\sup_{y>x}\frac{y\Phi'(y)}{\Phi(y)}$ (see
(\ref{eq:Delta2Equiv})). Let
$q_\Phi:=\lim_{x\rightarrow\infty}\frac{p_\Phi(x)}{p_\Phi(x)-1}=\frac{p_\Phi}{p_\Phi-1}>1$
with the convention $1/0=\infty$.

\begin{lemma}[cf. \citep{MR540367}, Prop.~2.b.5]
  \label{lem:PUpperEstim}
  For any $1\leq q<q_\Phi$, $L_{\Phi^*}$ has
  \textbf{\boldmath an upper $q$-estimate}, that is, there exists a
  constant $C_{q,\Phi^*}>0$ such that for any $n\in \NB$ and
  disjointly supported $\xi_1,...,\xi_n\in L_{\Phi^*}$ (i.e.
  $\xi_k=\xi_k\ind_{A_k}$ with $A_k\in\FC$ pairwise disjoint),
  \begin{align}
    \label{eq:UpperPEstim}
    \Biggl\|\sum_{k\leq n} \xi_k\Biggr\|_{(\Phi^*)}\leq C_{q,\Phi^*}
    \Biggl(\sum_{k\leq n}\|\xi_k\|_{(\Phi^*)}^{q}\Biggr)^{1/q}.
  \end{align}
\end{lemma}
\begin{proof}
  The case $q=1$ is trivial (we can take $C_{q,\Phi^*}=1$), and note
  that $1<q<\frac{p_\Phi}{p_\Phi-1}$ if and only if $q=\frac{p}{p-1}$
  for some $p\in (p_\Phi(x_0),\infty)$ and $x_0>0$. Fix such $q, p$
  and $x_0$. Then
  $\Psi(x):=\frac{\Phi(x_0)}{x_0}x\ind_{[0,x_0]}(x)+\Phi(x)\ind_{(x_0,\infty)}(x)$
  is a $\Delta_2$-Young function with $\Psi(x)=\Phi(x)$ for
  $x\geq x_0$, thus $L_\Psi=L_\Phi$ with equivalent norms (since
  $(\Omega,\FC,\PB)$ is a probability space, see \citep{MR1113700},
  Th.~V.1.3); hence there exists a $C>0$ such that
  \begin{equation}
    \label{eq:UpperQEstimProof1}
    C^{-1}\|\cdot\|_{(\Psi^*)}\leq \|\cdot\|_{(\Phi^*)}\leq
  C\|\cdot\|_{(\Psi^*)}.
  \end{equation}
  Moreover, $\Psi(x)>0$ for $x>0$ and
  $p_\Psi(0)=1\vee p_\Phi(x_0)=p_\Phi(x_0)<p<\infty$; in particular,
  for any $\lambda\geq 1$ and $x>0$,
  $\log\frac{\Psi(\lambda x)}{\Psi(x)}=\int_1^\lambda
  \frac{tx\Psi'(tx)}{\Psi(tx)}\frac{dt}{t}\leq p\log \lambda$, hence 
  \begin{equation}
    \label{eq:Delta2Alt2}
    \Psi(\lambda x)\leq \lambda^p\Psi(x)\text{ for }x>0,\,\lambda\geq 1.
  \end{equation}
  Therefore
  $1=\EB[\Psi(\eta/\|\eta\|_\Psi)]\leq
  \|\eta\|_\Psi^{-p}\EB[\Psi(\eta)]$
  for $0<\|\eta\|_\Psi\leq 1$, where the first equality is another
  consequence of $\Phi\in\Delta_2$. Hence we have
  \begin{equation}
    \label{eq:ProofKomlos1}
    \|\eta\|_\Psi\leq \EB[\Psi(\eta)]^{1/p}\text{ for all }\eta\in\BB_\Psi.
  \end{equation}

  Now if $\xi_k=\xi_k\ind_{A_k}\in L_{\Phi^*}=L_{\Psi^*}$ with
  $A_k\in\FC$ disjoint, then for any $\eta\in \BB_\Psi$,
  \begin{align*}
    \EB\Biggl[\Biggl(\sum_{k\leq n}\xi_k\Biggr)\eta\Biggr]
    &\leq\sum_{k\leq n}\|\xi_k\|_{(\Psi^*)}\|\eta\ind_{A_k}\|_\Psi
      \stackrel{\text{(\ref{eq:ProofKomlos1})}}\leq \sum_{k\leq n}\|\xi_k\|_{(\Psi^*)}\EB[\Psi(\eta)\ind_{A_k}]^{1/p}\\
    &\leq\Biggl(\sum_{k\leq n}\|\xi_k\|_{(\Psi^*)}^{q}\Biggr)^{1/q}\Biggl(\sum_{k\leq n}\EB[\Psi(\eta)\ind_{A_k}]\Biggr)^{1/p}
      \leq \Biggl(\sum_{k\leq n}\|\xi_k\|_{(\Psi^*)}^{q}\Biggr)^{1/q},
  \end{align*}
  since
  $\sum_{k\leq n}\EB[\Psi(\eta)\ind_{A_k}] \leq\EB[\Psi(\eta)]\leq 1$.
  Taking the supremum over $\eta\in\BB_\Psi$,
  \begin{align*}
    \frac1C    \Biggl\|\sum_{k\leq n}\xi_k\Biggr\|_{(\Phi^*)}
    \stackrel{\text{(\ref{eq:UpperQEstimProof1})}}\leq
    \Biggl\|\sum_{k\leq n}\xi_k\Biggr\|_{(\Psi^*)}
    \leq \Biggl(\sum_{k\leq n}\|\xi_k\|_{(\Psi^*)}^{q}\Biggr)^{\frac1q}
    \stackrel{\text{(\ref{eq:UpperQEstimProof1})}}\leq
    C\Biggl(\sum_{k\leq n}\|\xi_k\|_{(\Phi^*)}^{q}\Biggr)^{\frac1q}.
  \end{align*}
\end{proof}

\begin{corollary}
  \label{cor:DisjointCesaro}
  If $(\xi_n)_n$ is a norm bounded disjointly supported sequence in $L_{\Phi^*}$,
  then
  \begin{align*}
    \sup_n\left|\frac{\xi_1+\cdots+\xi_n}n\right|\in L_{\Phi^*}\quad\text{and}\quad
    \left\|\frac{\xi_1+\cdots+\xi_n}n\right\|_{\Phi^*}\rightarrow 0.
  \end{align*}
\end{corollary}
\begin{proof}
  Let $\xi_n=\xi_n\ind_{A_n}$ with $A_n\in\FC$ disjoint,
  $a:=\sup_n\|\xi_n\|_{(\Phi^*)}<\infty$, $1<q<q_\Phi$, and
  $C=C_{q,\Phi^*}$ as in Lemma~\ref{lem:PUpperEstim}. Put
  $\bar\xi_n:=\frac{\xi_1+\cdots+\xi_n}n$.  Then
  $\|\bar\xi_n\|_{(\Phi^*)}\leq a C\left(nn^{-q}\right)^{1/q}=aC
  n^{\frac1q-1}\rightarrow 0$.  Next, observe that
  $\|\sup_n|\bar\xi_n|\|_{(\Phi^*)}=\sup_N\|\sup_{n\leq
    N}|\bar\xi_n|\|_{(\Phi^*)}$ and
  \begin{align*}
    \sup_{n\leq N}|\bar\xi_n|=\sum_{k\leq N}\left(\sup_{n\leq N}|\bar\xi_n|\right)\ind_{A_k}
    =\sum_{k\leq N}\left(\sup_{k\leq n\leq N}\frac1n|\xi_k|\right)\ind_{A_k}=\sum_{k\leq N}\frac1k|\xi_k|,
  \end{align*}
  while $\left\|\sum_{k\leq N}\frac1k\xi_n\right\|_{(\Phi^*)} \leq a
  C\left(\sum_{k\leq N}\frac1{k^q}\right)^{1/q} \leq a
  C\left(\sum_{k=1}^\infty\frac1{k^q}\right)^{1/q}<\infty$, so
  $\sup_n|\bar\xi_n|\in L_{\Phi^*}$.
\end{proof}

Noting that
$\bar\xi_n=\frac{\xi_{n+1}+\cdots+\xi_{2n}}{n}=2\frac{\xi_1+\cdots+\xi_{2n}}{2n}-\frac{\xi_1+\cdots+\xi_n}n\in
\conv(\xi_n,\xi_{n+1},...)$, we get:
\begin{corollary}
  \label{cor:DisjConvConv1}
  Any norm bounded disjoint sequence $(\xi_n)_n$ in $L_{\Phi^*}$ has
  an order bounded and norm null sequence of forward convex
  combinations $\bar\xi_n\in \conv(\xi_k; k\geq n)$.
\end{corollary}

Since any subsequence of norm bounded disjoint sequence is again
bounded and disjoint, the same conclusion holds for any subsequence;
thus
\begin{corollary}
  \label{cor:DisjWeaklyNull}
  Any norm bounded disjoint sequence in $L_{\Phi^*}$ is
  $\sigma(L_{\Phi^*},L_{\Phi^*}')$-null.
\end{corollary}

\begin{remark}
  \label{rem:Alternative}
  The last two corollaries could be derived also from the fact that
  the dual of a Banach lattice $E$ has order continuous norm iff every
  norm bounded disjoint sequence in $E$ is weakly null
  (\cite[Th.~116.1]{MR704021} or \citep[Th.~2.4.14]{MR1128093}). In
  $L_{\Phi^*}$, $(\xi_n)_n$ is disjoint in the lattice sense iff it is
  disjointly supported, while $L_{\Phi^*}'=L_\Phi\oplus
  L_\infty^\medcirc$, where $L_\infty^\medcirc$ is the polar of
  $L_\infty\subset L_{\Phi^*}$ in $L_{\Phi^*}'$.
  The projections of $L_{\Phi^*}'$ onto $L_\Phi$ and onto
  $L_\infty^\medcirc$ are order continuous
  (e.g. \citep{MR0123907}). But $L_\infty^\medcirc$ is an AL space,
  hence has order continuous norm (regardless of $\Delta_2$;
  e.g. \citep[Th.~133.6]{MR704021}), thus $\Phi\in\Delta_2$ implies
  that ($\|\cdot\|_{(\Phi)}=\|\cdot\|_{L_{\Phi^*}'}|_{L_\Phi}$, hence)
  $\|\cdot\|_{L_{\Phi^*}'}$ is order continuous,
  so bounded disjoint sequences are weakly null.
\end{remark}

Now we can state the basic version of our Komlós type result.
\begin{theorem}
  \label{thm:Cesaro1}
  If $(\xi_n)_n$ is a norm bounded sequence in $L_{\Phi^*}$, converging
  in $\PB$ to some $\xi\in L_{\Phi^*}$, then there exists a
  subsequence $(\xi_{n_k})_k$ such that for any further subsequence
  $(\xi_{n_{k(i)}})_i$, the Cesàro means
  $\frac1N\sum_{k\leq N}\xi_{n_{k(i)}}$ converge in order to $\xi$,
  i.e.
  \begin{align}
    \label{eq:CesaroOrderConv1}
    \sup_N\Biggl|\frac1{N}\sum_{i\leq N}\xi_{n_{k(i)}}\Biggr|\in L_{\Phi^*}\text{ and }
    \frac1{N}\sum_{i\leq N}\xi_{n_{k(i)}}\stackrel{N}\rightarrow \xi\text{ a.s.}
  \end{align}
\end{theorem}

Here the original bounded sequence $(\xi_n)_n$ is supposed to converge
in $\PB$, which is needed to ensure that the Cesàro means themselves
of any subsequence converge in order. Without this a priori
assumption, we still have a slightly weaker conclusion.

\begin{theorem}
  \label{thm:Cesaro2}
  Any norm bounded sequence $(\xi_n)_n$ in $L_{\Phi^*}$ admits a
  subsequence $(\xi_{n_k})_k$ as well as $\xi\in L_{\Phi^*}$ such that
  for any subsequence $(\xi_{n_{k(i)}})_i$, the sequence of Cesàro
  means $\frac1N\sum_{k\leq N}\xi_{n_{k(i)}}$ has a subsequence order
  convergent to $\xi$, i.e. there is a sequence $(N_l)_l$ with
  $\sup_l\left|\frac1{N_l}\sum_{i\leq N_l}\xi_{n_{k(i)}}\right|\in
  L_{\Phi^*}$
  and
  $\frac1{N_l}\sum_{i\leq N_l}\xi_{n_{k(i)}} \rightarrow \xi\text{
    a.s.}$
\end{theorem}

\begin{lemma}[cf. \citep{MR1101668}]
  \label{lem:OrderBoundedSubseq}
  If $\xi_n\rightarrow 0$ in $\PB$ and if $(\Phi^*(\xi_n))_n$ is
  uniformly integrable, there exists a subsequence $(\xi_{n_k})_k$
  such that $\sup_k |\xi_{n_k}|\in L_{\Phi^*}$; in particular $\xi_n\rightarrow 0$ in $\tau(L_{\Phi^*},L_\Phi)$.
\end{lemma}
\begin{proof}
  The assumption implies $\EB[\Phi^*(\xi_n)]\rightarrow 0$, so there
  is a subsequence $(\xi_{n_k})_k$ such that
  $\sum_k\EB[\Phi^*(\xi_{n_k})]<\infty$. Noting that
  $\Phi^*(|\eta|\vee |\eta'|)
  =\Phi^*(\eta)\ind_{\{|\eta|>|\eta'|\}}+\Phi^*(\eta')\ind_{\{|\eta|\leq
    |\eta'|\}}\leq \Phi^*(\eta)+\Phi^*(\eta')$,
  a simple induction and the monotone convergence theorem show that
  \begin{align*}
    \EB\biggl[\Phi^*\biggl(\sup_k|\xi_{n_k}|\biggr)\biggr]\leq\lim_m    \EB\biggl[\Phi^*\biggl(\sup_{k\leq m}|\xi_{n_k}|\biggr)\biggr]
    \leq\lim_m \sum_{k\leq m}\EB[\Phi^*(\xi_{n_k})]\leq \sum_{k=1}^\infty \EB[\Phi^*(\xi_{n_k})]<\infty.
  \end{align*}
  Hence $\sup_k|\xi_{n_k}|\in L_{\Phi^*}$.  In particular,
  $\xi_{n_k}\rightarrow 0$ in $\tau(L_{\Phi^*},L_\Phi)$ by
  (\ref{eq:MackeyOrderInterval}). Since the assumptions on $(\xi_n)_n$
  are inherited to any subsequence, we deduce that every subsequence
  has a $\tau(L_{\Phi^*},L_\Phi)$-null subsequence; hence $(\xi_n)_n$
  itself is $\tau(L_{\Phi^*},L_\Phi)$-null.
\end{proof}

\begin{proof}[Proof of Theorems~\ref{thm:Cesaro1} and~\ref{thm:Cesaro2}]
  Let $(\xi_n)_n$ be a norm bounded sequence in $L_{\Phi^*}$, a
  fortiori bounded in $L_1$.  Komlós's theorem yields a subsequence,
  still denoted by $(\xi_n)_n$, and a $\xi\in L_1$ such that the
  Cesàro means of any further subsequence converges a.s. to $\xi$;
  then $\xi\in L_{\Phi^*}$ by Fatou's lemma. We can normalise
  $(\xi_n)_n$ so that $\xi=0$ and $\|\xi_n\|_{\Phi^*}\leq 1$
  ($\Leftrightarrow$
  $\EB[\Phi^*(\xi_n)]\leq 1$). Then the Kadec–Pe\l{}czy\'nski
  theorem (e.g. \citep[Lemma~5.2.8]{MR2192298}) applied to the bounded
  sequence $(\Phi^*(\xi_n))_n$ yields a subsequence $(\zeta_n)_n$ of
  $(\xi_n)$ as well as a disjoint sequence $(A_n)_n$ in $\FC$ such
  that $(\Phi^*(\zeta_n\ind_{A^c_n}))_n$ is uniformly integrable. Let
  $\zeta^r_n:=\zeta_n\ind_{A_n^c}$ and $\zeta^s_n:=\zeta_n\ind_{A_n}$
  so that $\zeta_n=\zeta^r_n+\zeta^s_n$.

  Now if the original sequence $(\xi_n)_n$ converges in $\PB$ (to $0$
  by the reduction above), then $(\zeta_n)_n\subset (\xi_n)_n$ as well
  as $(\zeta^r)_n$ are null in $\PB$. Since $(\Phi^*(\zeta^r_n))_n$ is
  uniformly integrable, Lemma~\ref{lem:OrderBoundedSubseq} yields a
  subsequence $(n_k)_k$ of positive integers such that
  $\eta':=\sup_k|\zeta^r_{n_k}|\in L_{\Phi^*}$. On the other hand,
  $(\zeta^s_n)_n$ (and any of its subsequence) is a norm bounded
  disjoint sequence, hence Corollary~\ref{cor:DisjointCesaro} shows
  that for any subsequence $(k(i))_i$,
  \begin{align*}
    \sup_N\left|\frac1N\sum_{i\leq N}\zeta_{n_{k(i)}}\right|
    \leq \sup_N\left|\frac1N\sum_{i\leq N}\zeta^r_{n_{k(i)}}\right|
    +\sup_N\left|\frac1N\sum_{i\leq N}\zeta^s_{n_{k(i)}}\right|
    \leq \eta'
    +\sup_N\left|\frac1N\sum_{i\leq N}\zeta_{n_{k(i)}}\right|
    \in L_{\Phi^*}.
  \end{align*}
  Since $\frac1N\sum_{i\leq N}\zeta_{n_{k(i)}}\rightarrow 0$ a.s.  by
  construction, we have Theorem~\ref{thm:Cesaro1}.

  Next, if $(\zeta_n)_n$ is not null in $\PB$, we can no longer hope
  for a ``universal bound'' for the regular part
  $(\zeta^r_n)_n$. However, once a subsequence $(n_k)_k$ is chosen we
  get
  \begin{align*}
    \bar\zeta_N:=\frac1N\sum_{k\leq N}\zeta_{n_k}=\frac1N\sum_{k\leq N}\zeta_{n_k}^r+\frac1N\sum_{k\leq N}\zeta_{n_k}^s=:\bar\zeta^r_N+\bar\zeta^s_N
    \rightarrow 0\text{ in }\PB,
  \end{align*}
  by the construction of $(\zeta_n)_n$. Again by
  Corollary~\ref{cor:DisjointCesaro}, $(\bar\zeta^s_N)_N$ is order
  bounded and norm null. In particular,
  $\bar\zeta_N^r=\bar\zeta_N-\bar\zeta^s_N\rightarrow 0$ in $\PB$, and
  $(\Phi^*(\bar\zeta^r_N))_N$ is uniformly integrable since $\Phi^*$
  is convex. Thus by Lemma~\ref{lem:OrderBoundedSubseq}, we find a
  subsequence $(N(i))_i$ such that $(\bar\zeta_{N(i)}^r)_i$, hence
  $(\bar\zeta_{N(i)})_i=(\bar\zeta_{N(i)}^r+\bar\zeta^s_{N(i)})_i$
  too, are order bounded.
\end{proof}

Since $(\bar\zeta_N^r)_N$ in the last paragraph is null in $\PB$ and
$(\Phi^*(\bar\zeta_N^r))_N$ is uniformly integrable, it is null in
$\tau(L_{\Phi^*},L_\Phi)$ by the last part of
Lemma~\ref{lem:OrderBoundedSubseq}. Thus we have also:

\begin{corollary}
  \label{cor:KomlosMackey}
  Any norm bounded sequence $(\xi_n)_n$ in $L_{\Phi^*}$ has a
  subsequence $(\xi_{n_k})_k$ and $\xi\in L_{\Phi^*}$ such that for
  any further subsequence $(n_{k(i)})_i$,
  $\frac1N\sum_{i\leq N}\xi_{n_{k(i)}}\rightarrow \xi$ in
  $\tau(L_{\Phi^*},L_\Phi)$.
\end{corollary}

At the moment, it is not clear if one can drop the assumption of
convergence in $\PB$ in Theorem~\ref{thm:Cesaro1}, or equivalently if
the Cesàro means in Theorem~\ref{thm:Cesaro2} are order bounded
without passing to a further subsequence. This question is left for a
future work. In applications, however, this point does not much
matter; since any norm bounded sequence in $L_{\Phi^*}$ (a fortiori
bounded in $L_1$) has an a.s. convergent sequence of forward convex
combinations by the usual Komlós theorem, and convex combinations of
convex combinations are convex combinations (cf. Cesàro means of
Cesàro means are not Cesàro means), we get the following utility grade
version of Theorem~\ref{thm:Cesaro1}.

\begin{corollary}
  \label{cor:ConvCombUseful}
  Any norm bounded sequence $(\xi_n)_n$ in $L_{\Phi^*}$ admits a
  sequence of forward convex combinations
  $\bar\xi_n\in\conv(\xi_k; k\geq n)$ as well as a $\xi\in L_{\Phi^*}$
  such that $\bar\xi_n\rightarrow\xi$ in order, i.e.
  $\sup_n|\bar\xi_n|\in L_{\Phi^*}$ and $\bar\xi_n\rightarrow \xi$
  a.s.
\end{corollary}

Regarding the property (\ref{eq:CProp}) of \citep{MR2648595}, we can
confirm that it is true for \emph{\bfseries bounded nets}, while the
boundedness cannot be dropped as noted in the introduction. Indeed, if
$(\xi_\alpha)_\alpha$ is a bounded net in $L_{\Phi^*}$ that converges
in $\sigma(L_{\Phi^*},L_\Phi)$ to $\xi\in L_{\Phi^*}$, then arguing as
in \citep[Lemma~6]{MR2648595}, one finds a sequence $(\alpha_n)$ of
indices as well as $\eta_n\in \conv(\xi_{\alpha_k};k\geq n)$ such that
$\eta_n\rightarrow \xi$ a.s. (this part is correct). Then
Corollary~\ref{cor:ConvCombUseful} yields
$\zeta_n\in\conv(\eta_k;k\geq n)\subset \conv(\xi_{\alpha_k};k\geq n)$
with $\sup_n|\zeta_n|\in L_{\Phi^*}$.

Finally, when $(\Omega,\FC,\PB)$ is atomless, these Komlós type
results characterise the $\Delta_2$-Orlicz spaces; in this case,
$\Phi\in\Delta_2$ if (and only if)
$\lim_n\|\xi\ind_{\{|\xi|>n\}}\|_{(\Phi)}=0$ for every $\xi\in L_\Phi$
(i.e. $\|\cdot\|_{(\Phi)}$ is order continuous on $L_\Phi$; see
\citep[Th.~133.4]{MR704021}).

\begin{theorem}
  \label{thm:Delta2Ch1}
  Suppose $(\Omega,\FC,\PB)$ is atomless, and let $\Phi$ be a (finite
  coercive) Young function (\textbf{not a priori assumed $\Delta_2$}). Then the
  following are equivalent:
  \begin{enumerate}
  \item $\Phi\in \Delta_2$;
  \item every norm bounded sequence in $L_{\Phi^*}$ has a subsequence
    with $\tau(L_{\Phi^*},L_\Phi)$-convergent Cesàro means;
  \item every norm bounded sequence in $L_{\Phi^*}$ has a
    $\sigma(L_{\Phi^*},L_\Phi)$-convergent sequence of forward convex
    combinations;
  \item every norm bound sequence in $L_{\Phi^*}$ has an order bounded
    sequence of forward convex combinations.
  \end{enumerate}
\end{theorem}

\begin{proof}
  (1) $\Rightarrow$ (4) is Corollary~\ref{cor:ConvCombUseful}, (1)
  $\Rightarrow$ (2) is Corollary~\ref{cor:KomlosMackey}, and (2)
  $\Rightarrow$ (3) and (4) $\Rightarrow$ (3) are clear. It remains to
  prove (3) $\Rightarrow$ (1).  Since $(\Omega,\FC,\PB)$ is atomless,
  $\Phi\not\in \Delta_2$ yields some $0\leq\zeta_0\in \BB_\Phi$ with
  $\lim_n\sup_{\eta\in
    \BB_{\Phi^*}}\EB[\zeta_0\eta\ind_{\{\zeta_0>n\}}]=\lim_n\|\zeta_0\ind_{\{\zeta_0>n\}}\|_{(\Phi)}>0$,
  so $\zeta_0\BB_{\Phi^*}$ is not uniformly integrable, hence there
  are $0\leq \eta_n\in \BB_{\Phi^*}$, \emph{\bfseries disjoint} sets
  $A_n\in\FC$, $n\geq 1$, and $\varepsilon>0$ such that
  $\EB[\zeta_0\eta_n\ind_{A_n}]\geq\varepsilon$ ($\forall n$). Then
  the
  bounded sequence $(\eta_n\ind_{A_n})_n$ has no
  $\sigma(L_{\Phi^*},L_\Phi)$-convergent forward convex combinations:
  if $\xi_n\in \conv(\eta_k\ind_{A_k}; k\geq n)$, $n\geq 1$, then
  $\xi_n\rightarrow0$ in $\PB$ since $A_n$ are disjoint, so the only
  possible $\sigma(L_{\Phi^*},L_\Phi)$-limit is $0$ by
  Corollary~\ref{cor:NullInPandWeakStarConv}, which is impossible
  since $\EB[\zeta_0\xi_n]\geq \inf_k\EB[\zeta_0\eta_k\ind_{A_k}]\geq
  \varepsilon$.
\end{proof}

\section{Closedness of Convex Sets}
\label{sec:ClosednessConti}

Now we deduce from Corollary~\ref{cor:ConvCombUseful} that
\begin{theorem}
  \label{thm:Main1}
  A \textbf{convex} subset $C\subset L_{\Phi^*}$ is
  $\sigma(L_{\Phi^*},L_\Phi)$-closed if and only if for every
  $\zeta\in L_{\Phi^*}$, the intersection $C\cap[-\zeta,\zeta]$ is
  closed in $L_0$ (i.e. order closed).
\end{theorem}
\begin{proof}
  The necessity is clear since $[-\zeta,\zeta]$ is closed in $L_0$ and
  $\tau(L_{\Phi^*},L_\Phi)|_{[-\zeta,\zeta]}=\tau_{L_0}|_{[-\zeta,\zeta]}$. For
  the sufficiency, it suffices that $C\cap \lambda\BB_{\Phi^*}$,
  $\lambda>0$, are closed in $L_0$
  (Proposition~\ref{prop:MackeyNorm}). Pick a sequence $(\xi_n)_n$ in
  $C\cap \lambda \BB_{\Phi^*}$ with $\xi_n\rightarrow\xi$ in $\PB$.
  Corollary~\ref{cor:ConvCombUseful} yields a sequence
  $\bar\xi_n\in\conv(\xi_k;k\geq n)\subset C$ (by convexity) with
  $\zeta:=\sup_n|\bar\xi_n|\in L_{\Phi^*}$, and
  $\bar\xi_n\rightarrow \xi$ a.s.  
  But $\lambda\BB_{\Phi^*}$ and $C\cap[-\zeta,\zeta]$ are
  $\tau_{L_0}$-closed, hence
  $\xi\in C\cap [-\zeta,\zeta]\cap \lambda\BB_{\Phi^*}$.
\end{proof}

To the best of our knowledge, this criterion for the weak*-closedness
is only known for \emph{\bfseries solid sets} (i.e.
$A\subset L_{\Phi^*}$ with $\zeta\in A$ and $|\xi|\leq |\zeta|$
$\Rightarrow$ $\xi\in A$); see \citep[Th.~4.20]{MR2011364}. But convex
functions with solid lower level sets are symmetric, so exclude all
non-trivial monotone convex functions, especially convex risk
measures. Also, since
$\sigma(L_{\Phi^*},L_\Phi)|_{[-\zeta,\zeta]}
\subset\tau_{L_0}|_{[-\zeta,\zeta]}
=\tau(L_{\Phi^*},L_\Phi)|_{[-\zeta,\zeta]}$,
$\zeta\in L_{\Phi^*}$ (by (\ref{eq:MackeyOrderInterval}) and
(\ref{eq:MackeyBall})), the condition is also equivalent to:
$C\cap [-\zeta,\zeta]$, $\zeta\in L_{\Phi^*}$, are
$\sigma(L_{\Phi^*},L_\Phi)$-closed.

\begin{remark}
  \label{rem:Gao}
  After our results were presented in Vienna Congress on Mathematical
  Finance, 12--14 September 2016
  (\href{https://fam.tuwien.ac.at/events/vcmf2016/}{https://fam.tuwien.ac.at/events/vcmf2016/}),
  and after a discussion with Niushan Gao, he and his collaborators
  \citep{GaoLeungXanthos2016} came up with their own proof of
  Theorem~\ref{thm:Main1}. They used a different technique which in
  our opinion will not yield a Komlós type theorem. The problem to get
  a Komlós type theorem was suggested by Hans Föllmer during the
  aforementioned Vienna conference.
\end{remark}

While the Mackey and weak* closed convex sets in the dual of a Banach
space are the same, \emph{\bfseries sequentially} Mackey closed convex
sets need not be (sequentially) weak* closed.  For instance,
$A=\{(\alpha_n)_n\in\ell_1:\alpha_1=\sum_{n\geq 2}\alpha_n\}$ is norm
closed but not sequentially weak* closed in $\ell_1=c_0'$ (see
\citep{banach32:_theor}), while since $\tau(\ell_1,c_0)$-convergent
\emph{\bfseries sequences} are norm convergent,
$A$ is sequentially $\tau(\ell_1,c_0)$-closed.  In our situation,
however, since
$\tau(L_{\Phi^*},L_\Phi)|_{[-\zeta,\zeta]}=\tau_{L_0}|_{[-\zeta,\zeta]}$,
$\zeta\in L_{\Phi^*}$, are metrisable, Theorem~\ref{thm:Main1} implies
that
\begin{corollary}
  \label{cor:SeqMackeyClosed}\mbox{}
  \textbf{Sequentially} $\tau(L_{\Phi^*},L_\Phi)$-closed convex sets
  in $L_{\Phi^*}$ are $\sigma(L_{\Phi^*},L_\Phi)$-closed.
\end{corollary}


Now the dual representation of proper convex functions on
$L_{\Phi^*}$, or equivalently the $\sigma(L_{\Phi^*},L_\Phi)$-lsc
($\Leftrightarrow$ $\tau(L_{\Phi^*},L_\Phi)$-lsc), is 
characterised as follows.
\begin{theorem}
  \label{thm:DualRep1}
  For a proper convex function $f$ on $L_{\Phi^*}$, the following are
  equivalent:
  \begin{enumerate}
  \item $f$ is $\sigma(L_{\Phi^*},L_\Phi)$-lsc, or equivalently
    $f(\xi)=\sup_{\eta\in L_\Phi}(\EB[\eta\xi]-f^*(\eta))$,
    $\xi\in L_{\Phi^*}$;
  \item $f$ is \textbf{sequentially} $\tau(L_{\Phi^*},L_\Phi)$-lsc;
  \item $f$ is $\tau_{L_0}$-lsc on every order interval
    $[-\zeta,\zeta]$ ($\zeta\in L_{\Phi^*}$), or equivalently order
    lsc: $f(\xi)\leq \liminf_nf(\xi_n)$ whenever $\xi_n\rightarrow\xi$
    a.s. and $(\xi_n)_n$ is order bounded in $L_{\Phi^*}$.
  \end{enumerate}
\end{theorem}

For the $\tau(L_{\Phi^*},L_\Phi)$-continuity, we have
\begin{theorem}
  \label{thm:MackeyConti}
  For any convex function $f:L_{\Phi^*}\rightarrow \RB$, the following
  are equivalent:
  \begin{enumerate}
  \item $f$ is $\tau(L_{\Phi^*},L_\Phi)$-continuous on $L_{\Phi^*}$;
  \item $f$ is \textbf{sequentially}
    $\tau(L_{\Phi^*},L_\Phi)$-continuous on $L_{\Phi^*}$;
  \item $f$ is \textbf{sequentially}
    $\tau(L_{\Phi^*},L_\Phi)$-continuous on \textbf{closed balls
      $\lambda\BB_{\Phi^*}$} ($\lambda>0$);
  \item $f$ is \textbf{sequentially}
    $\tau(L_{\Phi^*},L_\Phi)$-continuous on \textbf{order intervals};
  \item $f$ is $\tau_{L_0}$-continuous on \textbf{order intervals}, or
    equivalently order continuous, i.e. $f(\xi)=\lim_nf(\xi_n)$
    whenever $\xi_n\rightarrow\xi$ a.s. and $(\xi_n)_n$ is order
    bounded in $L_{\Phi^*}$.
  \end{enumerate}
\end{theorem}

\begin{proof}
  (1) $\Rightarrow$ (2) $\Rightarrow$ (3) $\Rightarrow$ (4) are
  trivial; (4) $\Leftrightarrow$ (5) since $\tau(L_{\Phi^*},L_\Phi)$
  coincides on order bounded sets with $\tau_{L_0}$. Suppose
  (5). Then, by Theorem~\ref{thm:DualRep1}, $f=f^{**}$, so by Moreau's
  theorem \citep{MR0160093}, it suffices that each
  $\Lambda_c:=\{\eta\in L_\Phi:f^*(\eta)\leq c\}$, $c\in\RB$, is
  $\sigma(L_\Phi,L_{\Phi^*})$-compact. By Young's inequality, for any
  $\lambda>0$, $\xi\in L_{\Phi^*}$ and $\eta\in\Lambda_c$,
  \begin{equation}
    \label{eq:ProofConti1}
    |\EB[\eta\xi\ind_A]|=\EB[\eta\xi\ind_A]\vee\EB[\eta(-\xi)\ind_A]\leq
  \frac1\lambda(f(\lambda\xi\ind_A)\vee f(-\lambda\xi\ind_A)+c)
  \end{equation}
  which implies that $\Lambda_c\xi$, $\xi\in L_{\Phi^*}$, are
  uniformly integrable, thus $\Lambda_c$ is
  $\sigma(L_\Phi,L_{\Phi^*})$-compact. For if $\Lambda_c\xi$ were not
  uniformly integrable, there would be $\varepsilon>0$, $A_n\in\FC$
  and $\eta_n\in\Lambda_c$ such that $\PB(A_n)\leq 2^{-n}$ and
  $|\EB[\eta_n\xi\ind_{A_n}]|\geq \varepsilon$; here note that
  $\EB[|\zeta|\ind_A]\geq2\varepsilon$ implies either
  $|\EB[\zeta\ind_{A\cap \{\zeta>0\}}]|\geq\varepsilon$ or
  $|\EB[\zeta\ind_{A\cap \{\zeta<0\}}]|\geq\varepsilon$ and
  $\PB(A\cap\{\zeta\gtrless0\})\leq \PB(A)$. But since
  $|\lambda\xi\ind_{A_n}|\leq\lambda |\xi|$ and
  $\lambda\xi\ind_{A_n}\rightarrow 0$ in $\PB$ for each $\lambda>0$,
  (5) and (\ref{eq:ProofConti1}) together with a diagonal argument
  show that $|\EB[\eta_n\xi\ind_{A_n}]|\rightarrow 0$, a
  contradiction.
\end{proof}

The property that $f$ is (sequentially) $\tau_{L_0}$-continuous on
every closed ball implies (via (5)) the Mackey continuity of $f$. The
converse implication holds for all finite convex functions if and only
if
$\tau(L_{\Phi^*},L_\Phi)|_{\BB_{\Phi^*}}=\tau_{L_0}|_{\BB_{\Phi^*}}$. Indeed,
seminorms generating the Mackey topology are finite valued Mackey
continuous convex functions. As we saw in Remark~\ref{rem:MackeyBDD1},
this is not the case if $\Phi(x)=x^2$; more generally, 
it fails whenever $\Phi^*\in\Delta_2$ (then $L_\Phi$ is
reflexive). Precisely when $\tau(L_{\Phi^*},L_\Phi)$ coincide with
$\tau_{L_0}$ on $\BB_{\Phi^*}$ is a subtle question which is left for
further investigation.

\begin{remark}
  \label{rem:Conti1}
  In the proof of (5) $\Rightarrow$ (1), we only used the facts that
  $f=f^{**}$ and $f|_{[-\zeta,\zeta]}$ is $\tau_{L_0}$-continuous at
  $0$, from which we derived that $f$ is
  $\tau(L_{\Phi^*},L_\Phi)$-continuous at $0$. Thus if $f$ is a priori
  supposed to be $\sigma(L_{\Phi^*},L_\Phi)$-lsc on $L_{\Phi^*}$ (or
  any of its equivalents in Theorem~\ref{thm:DualRep1}), and
  $f(\xi_0)<\infty$ (we can suppose $\xi_0=0$ by translation), the
  following remain equivalent: (1$'$) $f$ is
  $\tau(L_{\Phi^*},L_\Phi)$-continuous at $\xi_0$, (2$'$) $f$
  sequentially $\tau(L_{\Phi^*},L_\Phi)$-continuous at $\xi_0$, (3$'$)
  $f(\xi_0)=\lim_nf(\xi_n)$ whenever $\xi_n\rightarrow \xi_0$ in
  $\tau(L_{\Phi^*},L_\Phi)$ and $\sup_n\|\xi_n\|_{\Phi^*}<\infty$,
  (4$'$) the same but with $|\xi_n|\leq \zeta$ for some
  $\zeta\in L_{\Phi^*}^+$, (5$'$) the same but with
  $\xi_n\rightarrow\xi_0$ in $\PB$ and $|\xi_n|\leq \zeta$ for some
  $\zeta\in L_{\Phi^*}^+$.
\end{remark}

\subsection{Application to Monetary Utility Functions}
\label{sec:MonUtil}

In utility theory, \emph{\bfseries concave} functions
$u:L_{\Phi^*}\rightarrow \RB\cup\{-\infty\}$ satisfying the following
properties are called \emph{\bfseries monetary utility functions} (see
e.g. \citep{delbaen12:_monet_utilit_funct,MR2779313}):
\begin{align}
  \label{eq:Positivity}
  &u(0)=0; \, \xi\in L_{\Phi^*},\, \xi\geq 0\,\Rightarrow\, u(\xi)\geq 0;\\
  \label{eq:Cash}
  &a\in \RB,\, \xi\in L_{\Phi^*}\,\Rightarrow\, u(\xi+a)=u(\xi)+a.
\end{align}
Since $-u$ is a convex function, which is called a
\emph{\bfseries convex risk measure}, Theorems~\ref{thm:DualRep1}
and~\ref{thm:MackeyConti} with obvious change of sign characterise the
basic regularities of $u$ for the Mackey topology
$\tau(L_{\Phi^*},L_\Phi)$.  (\ref{eq:Positivity}) and (\ref{eq:Cash})
then give an even better description.

\begin{theorem}
  \label{thm:Utility1}
  A monetary utility function
  $u:L_{\Phi^*}\rightarrow\RB\cup\{-\infty\}$ is
  $\sigma(L_{\Phi^*},L_\Phi)$-upper semicontinuous (or what is the
  same, $\tau(L_{\Phi^*},L_\Phi)$-upper semicontinuous) if and only if
  it is continuous from above:
  \begin{align}
    \label{eq:ContiAbove}
    \xi_n\downarrow \xi\,\Rightarrow \, u(\xi)= \lim_n u(\xi_n).
  \end{align}
  In this case, the dual representation of $u$ can be written as
  \begin{equation}
    \label{eq:DualRepUtility}
    u(\xi)=\inf\{\EB_Q[\xi]+c(Q): c(Q)<\infty\},
  \end{equation}
  where $Q$ runs through probabilities absolutely continuous
  w.r.t. $\PB$ with $dQ/d\PB\in L_\Phi$, $c(Q)=(-u)^*(-dQ/d\PB)$ and
  $\EB_Q[\xi]=\EB[\xi dQ/d\PB]$.
\end{theorem}
\begin{proof}
  The necessity is clear from Theorem~\ref{thm:DualRep1} since
  $\xi_n\downarrow\xi$ implies $\xi_n\rightarrow \xi$ in order. For
  the sufficiency, we first show that
  (\ref{eq:Positivity})--(\ref{eq:ContiAbove}) imply that $u$ is
  monotone, i.e.
  \begin{equation}
    \label{eq:UtilMon}
    \xi,\eta\in L_{\Phi^*},\,    \xi\leq \eta \,\Rightarrow \, u(\xi)\leq u(\eta)
  \end{equation}
  We can suppose $u(\xi)=0$ thanks to (\ref{eq:Cash}). For each
  $\varepsilon\in (0,1)$, let
  $\alpha_\varepsilon=(1-\varepsilon)/\varepsilon$ so that
  $\zeta_\varepsilon:=\eta+\varepsilon\xi^-+\alpha_\varepsilon(\eta+\varepsilon\xi^--\xi)\geq
  0$.
  Putting
  $\lambda_\varepsilon:=\alpha_\varepsilon/(1+\alpha_\varepsilon)\in
  (0,1)$,
  we have
  $\eta+\varepsilon\xi^-=\lambda_\varepsilon
  \xi+(1-\lambda_\varepsilon)\zeta_\varepsilon$,
  hence by the concavity,
  $u(\eta+\varepsilon\xi^-)\geq \lambda_\varepsilon
  u(\xi)+(1-\lambda_\varepsilon)u(\zeta_\varepsilon)\geq 0$.
  Then (\ref{eq:ContiAbove}) shows that
  $u(\eta)=\lim_nu(\eta+n^{-1}\xi^-)\geq 0=u(\xi)$.  Now by
  Theorem~\ref{thm:DualRep1} applied to the \emph{\bfseries convex}
  function $-u$, the $\sigma(L_{\Phi^*},L_\Phi)$-upper semicontinuity
  of $u$ is equivalent to the property that
  $u(\xi)\geq \limsup_n u(\xi_n)$ whenever $\xi_n\rightarrow \xi$
  a.s. and $(\xi_n)_n$ is order bounded in $L_{\Phi^*}$; given the
  monotonicity (\ref{eq:UtilMon}) of $u$, this is equivalent to
  (\ref{eq:ContiAbove}). That the dual representation of $f=-u$
  together with (\ref{eq:Positivity}) and (\ref{eq:Cash}) yields
  (\ref{eq:DualRepUtility}) is standard.
\end{proof}

Note that if $u$ is finite valued ($\RB$-valued),
(\ref{eq:Positivity}) and (\ref{eq:Cash}) still imply
(\ref{eq:UtilMon}) without assuming (\ref{eq:ContiAbove}). For
$\varepsilon\mapsto u(\eta+\varepsilon\xi^-)$ is continuous as a
finite valued convex function on $\RB$. One can easily see also that
any monetary utility function that is
$\tau(L_{\Phi^*},L_\Phi)$-continuous at $0$ is finite valued.  For
such $u$, Theorem~\ref{thm:MackeyConti} yields that
\begin{theorem}
  \label{thm:UtilityMackey}
  A monetary utility function $u:L_{\Phi^*}\rightarrow\RB$ is
  $\tau(L_{\Phi^*},L_\Phi)$-continuous if (and only if) it is
  continuous from below, i.e. $\xi_n\uparrow \xi$ $\Rightarrow$
  $u(\xi)=\lim_nu(\xi_n)$.
\end{theorem}
\begin{proof}
  Given that $u$ is finite, monotone and convave, the continuity from
  below implies the continuity from above. For if
  $\xi_n\downarrow\xi$, then
  $u(\xi)\geq \frac12 u(\xi_n)+\frac12 u(2\xi-\xi_n)$ by the
  concavity, so the continuity from below and the monotonicity imply
  $0\leq u(\xi_n)-u(\xi)\leq u(\xi)-u(2\xi-\xi_n)\downarrow 0$ since
  $2\xi-\xi_n\uparrow \xi$.
  In
  particular, $u$ is $\sigma(L_{\Phi^*},L_\Phi)$-usc. On the other
  hand, again by the monotonicity, the continuity of $u$ from below is
  equivalent to the property that $u(\xi)=\lim_nu(\xi_n)$ whenever
  $\xi_n\rightarrow \xi$ a.s. and $(\xi_n)_n$ is order bounded in
  $L_{\Phi^*}$. The result now follows from
  Theorem~\ref{thm:MackeyConti}.
\end{proof}

\appendix

\section*{Appendix}
\label{sec:AppA}
\addcontentsline{toc}{section}{Appendix}

\begin{proof}[Proof of Proposition~\ref{prop:ContiLinfty1}]
  Only (3) $\Rightarrow$ (1) deserves a proof. (3) implies, by
  Proposition~\ref{prop:LinftyLSC1}, $f=f^{**}$, and
  $|\EB[\eta\ind_A]|\leq \frac1n\left(f(n\ind_A)\vee
    f(-n\ind_A)+c\right)$
  for $A\in\FC$ and $\eta\in L_1$ with $f^*(\eta)\leq c$ by Young's
  inequality; thus (3) implies that $\{\eta\in L_1: f^*(\eta)\leq c\}$
  is uniformly integrable, hence $\sigma(L_1,L_\infty)$-compact by the
  Dunford-Pettis theorem. Now Moreau's theorem \citep{MR0160093} shows
  that $f$ is $\tau(L_\infty,L_1)$-continuous.
\end{proof}

\begin{proof}[Proof of Lemma~\ref{lem:CompactSets}]
  For each $\xi\in L_{\Phi^*}$, $\eta\mapsto \eta\xi$ continuously
  maps $(L_\Phi,\sigma(L_\Phi,L_{\Phi^*}))$ into
  $(L_1,\sigma(L_1,L_\infty))$ since $\xi\zeta\in L_1$,
  $\forall \zeta\in L_\infty$.  Thus if $A$ is relatively
  $\sigma(L_\Phi,L_{\Phi^*})$-compact, its image $A\xi$ is relatively
  weakly compact in $L_1$, i.e. uniformly integrable. Conversely, if
  $A\xi$, $\xi\in L_{\Phi^*}$, are uniformly integrable, then
  $c_\xi:=\sup_{\eta\in A}\EB[|\eta\xi|]<\infty$ for each
  $\xi\in L_{\Phi^*}$, so $A$ is pointwise bounded in the algebraic
  dual $L_{\Phi^*}^\#$ of $L_{\Phi^*}$, and $A$ is relatively
  $\sigma(L_1,L_\infty)$-compact in $L_1$. Thus if
  $(\eta_\alpha)_\alpha$ is a net in $A$ with the pointwise limit
  $f(\xi)=\lim_\alpha\EB[\eta_\alpha\xi]$ in $L_{\Phi^*}^\#$, there is
  a unique $\eta_0\in L_1$ such that
  $f|_{L_\infty}(\xi)=\EB[\eta_0\xi]$ for $\xi\in L_\infty$. Then for
  each $\xi\in L_{\Phi^*}$,
  $\EB[|\eta_0\xi|]=\sup_n\EB[\eta_0\xi\ind_{\{|\xi|\leq n\}}
  \mathrm{sgn}(\eta_0\xi)]
  =\sup_nf(\xi\ind_{\{|\xi|\leq n\}}\mathrm{sgn}(\eta_0\xi)) \leq
  c_\xi$,
  hence $\eta_0\in L_\Phi$, while
  $|f(\xi)-f(\xi\ind_{\{|\xi|\leq
    n\}})|=|f(\xi\ind_{\{|\xi|>n\}})|\leq
  \sup_{\eta\in A}\EB[|\eta\xi|\ind_{\{|\xi|>n\}}] \rightarrow 0$
  since $A\xi$ is uniformly integrable; hence
  $f(\xi)=\EB[\eta_0\xi]$. Therefore $A$ is pointwise bounded and its
  $\sigma(L_{\Phi^*}^\#,L_{\Phi^*})$-closure in $L_{\Phi^*}^\#$ lies
  in $L_\Phi$; hence $A$ is relatively
  $\sigma(L_\Phi,L_{\Phi^*})$-compact.
\end{proof}

\small
\def\cprime{$'$}


\begin{thebibliography}{21}
\providecommand{\natexlab}[1]{#1}
\providecommand{\url}[1]{\texttt{#1}}
\providecommand{\urlprefix}{URL }
\providecommand{\eprint}[2][]{\url{#2}}

\bibitem[{Albiac and Kalton(2006)}]{MR2192298}
Albiac, F. and N.~J. Kalton (2006): Topics in {B}anach space theory,
  \textit{Graduate Texts in Mathematics}, vol. 233.
\newblock Springer, New York.

\bibitem[{Aliprantis and Burkinshaw(2003)}]{MR2011364}
Aliprantis, C.~D. and O.~Burkinshaw (2003): Locally solid {R}iesz spaces with
  applications to economics, \textit{Mathematical Surveys and Monographs}, vol.
  105.
\newblock American Mathematical Society, Providence, RI, 2nd ed.

\bibitem[{And{\^o}(1960)}]{MR0123907}
And{\^o}, T. (1960): Linear functionals on {O}rlicz spaces.
\newblock \textit{Nieuw Arch. Wisk. (3)} \textbf{8}, 1--16.

\bibitem[{Banach(1932)}]{banach32:_theor}
Banach, S. (1932): Théorie des opérations linéaires, \textit{Monografie
  Matematyczne}, vol.~1.
\newblock Instytut Matematyczny Polskiej Akademi Nauk, Warszawa.
\newblock Reprinted by Chelsea Publishing Co., 1955.

\bibitem[{Biagini and Frittelli(2009)}]{MR2648595}
Biagini, S. and M.~Frittelli (2009): On the extension of the {N}amioka-{K}lee
  theorem and on the {F}atou property for risk measures.
\newblock In: Optimality and risk---modern trends in mathematical finance,
  Springer, Berlin, pp. 1--28.

\bibitem[{Delbaen(2009)}]{MR2648597}
Delbaen, F. (2009): Differentiability properties of utility functions.
\newblock In: Optimality and risk---modern trends in mathematical finance,
  Springer, Berlin, pp. 39--48.

\bibitem[{Delbaen(2012)}]{delbaen12:_monet_utilit_funct}
Delbaen, F. (2012): Monetary Utility Functions, \textit{Osaka University CSFI
  Lecture Notes Series}, vol.~3.
\newblock Osaka University Press.

\bibitem[{F{\"o}llmer and Schied(2011)}]{MR2779313}
F{\"o}llmer, H. and A.~Schied (2011): Stochastic finance.
\newblock Walter de Gruyter \& Co., Berlin, 3rd ed.
\newblock An introduction in discrete time.

\bibitem[{Gao et~al.(2016)Gao, Leung and Xanthos}]{GaoLeungXanthos2016}
Gao, N., D.~H. Leung and F.~Xanthos (2016): The dual representation problem of
  risk measures.
\newblock arXiv:1610.08806.

\bibitem[{Gao and Xanthos(2015)}]{Gao_Xanthos15:_c}
Gao, N. and F.~Xanthos (2015): On the {$C$}-property and
  {$w^*$}-representations of risk measures.
\newblock arXiv:1511.03159.

\bibitem[{Grothendieck(1973)}]{MR0372565}
Grothendieck, A. (1973): Topological vector spaces.
\newblock Gordon and Breach Science Publishers, New York.
\newblock Translated from the French by Orlando Chaljub, Notes on Mathematics
  and its Applications.

\bibitem[{Jouini et~al.(2006)Jouini, Schachermayer and Touzi}]{MR2277714}
Jouini, E., W.~Schachermayer and N.~Touzi (2006): Law invariant risk measures
  have the {F}atou property.
\newblock In: Advances in mathematical economics. {V}ol. 9, \textit{Adv. Math.
  Econ.}, vol.~9, Springer, Tokyo, pp. 49--71.

\bibitem[{Koml{\'o}s(1967)}]{MR0210177}
Koml{\'o}s, J. (1967): A generalization of a problem of {S}teinhaus.
\newblock \textit{Acta Math. Acad. Sci. Hungar.} \textbf{18}, 217--229.

\bibitem[{Lindenstrauss and Tzafriri(1979)}]{MR540367}
Lindenstrauss, J. and L.~Tzafriri (1979): Classical {B}anach spaces. {II},
  \textit{Ergebnisse der Mathematik und ihrer Grenzgebiete [Results in
  Mathematics and Related Areas]}, vol.~97.
\newblock Springer-Verlag, Berlin-New York.
\newblock Function spaces.

\bibitem[{Meyer-Nieberg(1991)}]{MR1128093}
Meyer-Nieberg, P. (1991): Banach lattices.
\newblock Universitext. Springer-Verlag, Berlin.

\bibitem[{Moreau(1964)}]{MR0160093}
Moreau, J.-J. (1964): Sur la fonction polaire d'une fonction semi-continue
  sup\'erieurement.
\newblock \textit{C. R. Acad. Sci. Paris} \textbf{258}, 1128--1130.

\bibitem[{Nowak(1988)}]{MR1101668}
Nowak, M. (1988): On the order structure of {O}rlicz lattices.
\newblock \textit{Bull. Polish Acad. Sci. Math.} \textbf{36}, 239--249 (1989).

\bibitem[{Ostrovskii(2001)}]{MR1901532}
Ostrovskii, M.~I. (2001): Weak* sequential closures in {B}anach space theory
  and their applications.
\newblock In: General topology in {B}anach spaces, Nova Sci. Publ., Huntington,
  NY, pp. 21--34.

\bibitem[{Ostrovskii(2011)}]{MR2943053}
Ostrovskii, M.~I. (2011): Weak* closures and derived sets in dual {B}anach
  spaces.
\newblock \textit{Note Mat.} \textbf{31}, 129--138.

\bibitem[{Rao and Ren(1991)}]{MR1113700}
Rao, M.~M. and Z.~D. Ren (1991): Theory of {O}rlicz spaces, \textit{Monographs
  and Textbooks in Pure and Applied Mathematics}, vol. 146.
\newblock Marcel Dekker, Inc., New York.

\bibitem[{Zaanen(1983)}]{MR704021}
Zaanen, A.~C. (1983): Riesz spaces. {II}, \textit{North-Holland Mathematical
  Library}, vol.~30.
\newblock North-Holland Publishing Co., Amsterdam.

\end{thebibliography}
\end{document}